\documentclass[a4paper,11pt]{article}
\usepackage[top=3cm,bottom=3cm,left=2.5cm,right=2.5cm]{geometry}
\usepackage[T1]{fontenc}
\usepackage[utf8]{inputenc}
\usepackage{lmodern,microtype}
\usepackage[english]{babel}
\usepackage{amsmath,amssymb,amsthm,mathtools}
\usepackage{enumitem,aliascnt,cite,needspace}
\usepackage[hidelinks,hypertexnames=false]{hyperref}
\usepackage[nameinlink,capitalize,noabbrev]{cleveref}
\setlist[enumerate]{leftmargin=2.2em,itemsep=.2em,topsep=.3em}
\setlist[itemize]{leftmargin=2em,itemsep=.2em,topsep=.3em,beginpenalty=10000,midpenalty=10000}
\allowdisplaybreaks
\newtheorem{lemma}{Lemma}[section]
\newaliascnt{theorem}{lemma}
\newtheorem{theorem}[theorem]{Theorem}
\aliascntresetthe{theorem}
\newaliascnt{proposition}{lemma}
\newtheorem{proposition}[proposition]{Proposition}
\aliascntresetthe{proposition}
\newaliascnt{corollary}{lemma}
\newtheorem{corollary}[corollary]{Corollary}
\aliascntresetthe{corollary}
\newcommand{\cF}{\mathcal F}
\newcommand{\cD}{\mathcal D}
\newcommand{\cR}{\mathcal R}
\hypersetup{pdftitle={A cubic chromatic bound for capped graphs},pdfauthor={Hongyang Wang},pdfsubject={Graph coloring and capped ordered graphs}}
\input{glyphtounicode}
\title{A cubic chromatic bound for capped graphs}
\author{Hongyang Wang\thanks{Email address: \href{mailto:hwanghj@connect.ust.hk}{hwanghj@connect.ust.hk}}\\[.4em]
\normalsize Department of Mathematics,\\
\normalsize Hong Kong University of Science and Technology,\\
\normalsize Clear Water Bay, Hong Kong}
\date{}

\begin{document}
\maketitle

\begin{abstract}
An ordered graph is capped if, for any four vertices $a<b<c<d$, the
edges $ac$ and $bd$ imply the edge $ad$. We prove that
every nonempty capped graph $G$ satisfies
$\chi(G)\leq\min\{3^{\omega(G)-1},18\omega(G)^3\}$, where $\chi(G)$
and $\omega(G)$ are its chromatic number and clique number. This proves
a conjecture of Davies, Krawczyk, McCarty and Walczak~\cite{DaviesEtAl2023}.
\end{abstract}



\section{Introduction}\label{sec:introduction}

All graphs are finite, simple and undirected. For a graph $G$, write
$V(G)$ and $E(G)$ for its vertex set and edge set. The subgraph induced
by $X\subseteq V(G)$ is $G[X]$, and $G-v=G[V(G)\setminus\{v\}]$.
A spanning subgraph of $G$ has vertex set $V(G)$.
The neighborhood of a vertex $v$ is $N_G(v)$, its degree is
$|N_G(v)|$, and $N_G(X)=\bigcup_{v\in X}N_G(v)$.
Disjoint vertex sets are
\emph{anticomplete} if there is no edge between them. A \emph{clique}
is a set of pairwise adjacent vertices, and a \emph{stable set} is a
set of pairwise nonadjacent vertices. The maximum size of a clique is
$\omega(G)$. A triangle is a clique of size three, and a graph is
\emph{triangle-free} if it has no triangle. A \emph{forest} has no
cycle, and a \emph{bipartite graph} has a partition into two stable
sets. An \emph{induced path} has no edges between nonconsecutive vertices.

A proper coloring assigns different colors to adjacent vertices, and
$\chi(G)$ is the minimum number of colors needed. A \emph{palette} is
the set of available colors. A color class is the set of vertices
assigned a particular color. A set of colored vertices is
\emph{monochromatic} if all its vertices have the same color.
We use the convention $\chi(G)=\omega(G)=0$ when $G$ has no vertices.

A class of graphs is \emph{$\chi$-bounded} if there is a function $f$
such that $\chi(H)\leq f(\omega(H))$ for every nonempty induced
subgraph $H$ of a graph in the class.
It is \emph{polynomially $\chi$-bounded} if $f$ can be
chosen to be a polynomial. For background on $\chi$-boundedness,
see the survey of Scott and Seymour~\cite{ScottSeymour2020}.

An \emph{ordered graph} is a graph together with a fixed linear order
$v_1<\cdots<v_n$ on its vertices. For vertex sets $X,Y$, the notation $X<Y$ means that
every vertex of $X$ precedes every vertex of $Y$; the same notation is
used when either set is a single vertex. Throughout, left and right refer
to this vertex order. Two edges \emph{cross} if
their endpoints occur in the order $a<b<c<d$ and the edges are $ac,bd$.
An ordered graph is \emph{capped} if every such pair of crossing edges
has the edge $ad$, called its \emph{cap}.
All induced subgraphs carry the restricted order. Taking an induced
subgraph or reversing the order preserves cappedness.

Capped graphs were studied by Davies, Krawczyk, McCarty and
Walczak~\cite{DaviesEtAl2023} in connection with visibility graphs of
polygons. The visibility graph of a simple polygon has the polygon vertices
as its vertices; two are adjacent if the segment joining them lies
in the closed region bounded by the polygon. The vertex order
follows the polygon boundary.

Capped graphs are also called terrain-like graphs; their enumeration was
studied by Froese and Renken~\cite{FroeseRenken2024}.
Compact representations of these graphs were studied by Cardinal and
Yuditsky~\cite{CardinalYuditsky2025} and Cardinal and
Sharir~\cite{CardinalSharir2026}.

\Needspace{7\baselineskip}
Davies, Krawczyk, McCarty and Walczak~\cite{DaviesEtAl2023} proved that every
capped graph with clique number $w\geq2$ has chromatic number at most
$4^w(w-1)$. They conjectured that the class of capped graphs is
polynomially $\chi$-bounded~\cite{DaviesEtAl2023}.
We prove the following bound.

\begin{theorem}\label{thm:main}
Every nonempty capped graph $G$ satisfies
$\chi(G)\leq\min\{3^{\omega(G)-1},18\omega(G)^3\}$.
\end{theorem}

If $G$ has at least one edge, the proof gives the slightly stronger
polynomial bound $18(\omega(G)^3-\omega(G))$. 
The exponential bound gives three colors when $\omega(G)=2$, and this
is best possible: the five-cycle $v_1v_2v_3v_4v_5v_1$, ordered as
$v_1<v_2<v_3<v_4<v_5$, has no crossing edges and requires three colors.

An \emph{ordered hole} is an induced cycle $c_1c_2\cdots c_kc_1$
with $k\geq4$ and $c_1<\cdots<c_k$.
Davies, Krawczyk, McCarty and Walczak~\cite{DaviesEtAl2023}
noted that three-colorability of triangle-free capped graphs
containing no ordered hole would improve their coloring bounds.
\Cref{cor:triangle-free} proves three-colorability for all
triangle-free capped graphs.

The paper is organized as follows. \Cref{sec:boundary} proves a
three-color extension lemma for triangle-free capped graphs.
In \cref{sec:ranks}, we define edge ranks and obtain the exponential
bound. \Cref{sec:cubic} develops the forest construction and uses it
to prove the cubic bound.

\section{Triangle-free capped graphs}\label{sec:boundary}

An ordered partition $L<X<B$ consists of three disjoint sets
covering the vertex set, with all vertices of $L$ first, then those
of $X$, and then those of $B$. We call $X$ the interior and $L,B$
the left and right boundaries. A prescribed coloring of $L\cup B$
is proper on the induced graph on that set. To extend it means
to keep every prescribed color.

The three-color extension lemma below requires $L$ to be monochromatic
and $B$ to contain at most one vertex. We first prove the following
bipartiteness lemma.

\begin{lemma}\label{lem:bipartite}
Let $G$ be a capped ordered graph with a partition $V(G)=S\cup U$.
Suppose that every vertex of $S$ has only greater neighbors. Suppose
also that, for every edge $xy$ with $x<y$ and $x,y\in U$, the vertex
$x$ is the least neighbor of $y$ in $G$. Then $G$ is bipartite.
\end{lemma}

\begin{proof}
All the hypotheses are preserved under induced restriction, so we
induct on $|V(G)|$. The assertion is immediate if $G$ has no vertices.
If $S$ is empty,
the greatest vertex has degree at most one: two neighbors $x<y$
would contradict the least-neighbor condition on the edge from $y$
to that vertex. Delete the vertex, color the remaining graph with
two colors by induction, and restore it with an available color.

Suppose $S$ is nonempty, let $s=\max S$, and let
$T=\{v\in V(G):s<v\}$. If $G[T]$ has an edge, choose the greatest vertex $y$ that is
nonisolated in $G[T]$, and let $x\in T$ be a neighbor of $y$. Then $x<y$, and $y$
has no greater neighbor in $G$. Since $x,y\in U$, the least neighbor
of $y$ is $x>s$. Thus $y$ has no neighbor in $S$. It has no second
neighbor in $U$ either, since each such neighbor would also have to
be its least neighbor. Hence $y$ has degree one, and deletion and
restoration again complete the coloring.

It remains to consider the case that $T$ is stable. All neighbors
of $s$ are in $T$. If $u<v$ are neighbors of $s$, then
$N_G(u)\subseteq N_G(v)$. Indeed, any $t\in N_G(u)\setminus\{s\}$
is before $s$, and the edges $tu,sv$ at $t<s<u<v$ give $tv$ by
cappedness.

Color $G-s$ with two colors by induction. If $s$ is isolated, give
it either color. Otherwise let $d=\max N_G(s)$. For any
$u\in N_G(s)\setminus\{d\}$ that is not isolated in $G-s$, choose
$t\in N_G(u)\setminus\{s\}$. The inclusion just proved gives $td$,
so the two-edge path $u,t,d$ in $G-s$ forces $u$ and $d$ to have the
same color. Give every neighbor of $s$ that is isolated in $G-s$
the color of $d$ as well. All neighbors of $s$ now have one color,
so the other color can be assigned to $s$.
\end{proof}

\begin{lemma}\label{lem:boundary}
Let $G$ be a triangle-free capped ordered graph with an ordered
partition $L<X<B$, where $L$ is stable and $B$ is empty or a singleton
$\{b\}$. Every proper coloring of $G[L\cup B]$ with a palette of
three colors, monochromatic on $L$, extends to a proper coloring of
$G$ with the same palette.
\end{lemma}

\begin{proof}
We induct on $|X|$, for all graphs and boundary colorings in the
statement. If $X$ is empty, there is nothing to prove.

Let $\alpha$ be the prescribed color on $L$; if $L$ is empty, choose
any color of the palette as $\alpha$. Let $Q$ be the other two colors,
and let $T=N_G(L)\cap X$. Thus $T$ is exactly the set of interior
vertices with a neighbor in $L$.

If $T$ is empty, let $x$ be the least vertex of $X$. Give $x$ a color
different from the color of $b$ when $b$ exists, and any color
otherwise. In the induced graph $G[X\cup B]$, use $\{x\}$ as the
left boundary, $X\setminus\{x\}$ as the interior, and $B$ as the
right boundary. The prescribed boundary coloring is proper, and the
interior is smaller. The induction hypothesis colors this graph.
Together with the original coloring of $L$, this colors $G$: there
are no edges from $L$ to $X$, and the prescribed edges from $L$ to
$B$ are already proper.

Assume for the rest of the proof that $T$ is nonempty. Say that a
vertex \emph{contacts} a set if it has a neighbor in that set. We
need the following three facts about these contacts.

\begin{enumerate}[label=\textup{(\roman*)}]
\item\label{contact:first}
If $x\in T$, $x<y$ with $y\in X\cup B$, and $xy\in E(G)$, then
$y$ has no neighbor $z\in X$ with $z<x$. For otherwise choose an
$L$-neighbor $a$ of $x$. The order $a<z<x<y$ and the edges $ax,zy$
give $ay$ by cappedness. Then $a,x,y$ form a triangle.

\item\label{contact:gaps}
Each component $D$ of $G[X\setminus T]$ is contained in a single gap
of $T$: before its first vertex, between two consecutive vertices of
$T$, or after its last vertex. Moreover, $D$ has at most one contact
in $T$ to its left, and that contact, if present, is the last vertex
of $T$ before the gap.

To prove the first assertion, an edge $xz$ of $G[X\setminus T]$
cannot satisfy $x<y<z$ for $y\in T$. An $L$-neighbor $a$ of $y$
would give $az$ from the crossing edges $ay,xz$, contrary to
$z\notin T$. A path in $G[X\setminus T]$ therefore cannot pass from
one gap to another. For the second assertion, suppose that a vertex
$x\in T$ contacts $D$ and another vertex $y\in T$ is between $x$
and $D$. Choose $d\in D$ adjacent to $x$ and an $L$-neighbor $a$ of
$y$. The order $a<x<y<d$ gives $ad$ from $ay,xd$, contrary to
$d\notin T$.

\item\label{contact:right}
If $b$ exists and has no neighbor in $L$, then every neighbor of $b$
in $X$ is at or after every vertex of $T$. Indeed, if $ub$ is an
edge with $u<x$ for some $x\in T$, an $L$-neighbor $a$ of $x$ gives
$ab$ from $ax,ub$ at $a<u<x<b$, a contradiction. In particular, a
component $D$ of $G[X\setminus T]$ contacting $b$ is after all of
$T$. Its contacts to the right among $T\cup B$ are then exactly
$\{b\}$.
\end{enumerate}

To color $T$ with the two colors in $Q$, retain $b$ if it exists
and its prescribed color is different from $\alpha$; otherwise
omit it for the moment. Let $P$ consist of $T$ and the retained
vertex $b$, if any. Construct an ordered
graph $A$ with vertex set
$P\cup\{s_D:D\in\cD\}$, where $\cD$ is the set of components of
$G[X\setminus T]$. Each $s_D$ is a new symbol placed at the position
of $\min D$ in the order. Keep every original edge with both ends
in $P$. Join $s_D$ to $t\in P$ precisely when every vertex of $D$
is before $t$ and $t$ contacts $D$. There are no other edges.
Thus symbols have only greater neighbors, and no two symbols are
adjacent. The symbols will make all right contacts of a component
receive one color, so that induction can later be applied to that
component in the reversed order.

The components need not be intervals in the order, so we check
cappedness of $A$ using their connectedness. Associate a set
$C_x$ to each vertex $x$ of $A$: $C_x=\{x\}$ for $x\in P$, and
$C_{s_D}=D$ for a symbol. These sets are pairwise disjoint and each
is connected in $G$.

Take edges $xz,yt$ of $A$ with $x<y<z<t$. Their greater endpoints
$z,t$ are in $P$. Also $z\in T$, since a retained $b$ is the last
vertex. By \ref{contact:gaps}, every vertex of $C_x\cup C_y$ is
before $z$. Choose witnessing original edges $cz$ and $dt$ with
$c\in C_x$ and $d\in C_y$. If $c<d$, then
$c<d<z<t$ gives $ct$. If $d<c$, the order in $A$
gives $\min C_x<\min C_y\leq d<c<z$. A path in $G[C_x]$ from its
minimum to $c$ has an edge $ef$ with $e<d<f<z$, since $d\notin C_x$.
The cap of $ef$ and $dt$ is $et$. Equality $c=d$ is impossible by
disjointness. In both cases $C_x$ contacts $t$; all of $C_x$ is before
$t$, so $xt$ is an edge of $A$ by its definition.

For an edge $xy$ of $A$ with $x<y$ and $x,y\in P$, we have $x\in T$.
There is no neighbor of $y$ in $P$ before $x$, by
\ref{contact:first}. Nor can a symbol $s_D<x$ be adjacent to $y$:
\ref{contact:gaps} puts the whole of $D$ before $x$, and a witnessing
edge from $D$ to $y$ would contradict \ref{contact:first}. Therefore
$x$ is the least neighbor of $y$ in $A$.
\Cref{lem:bipartite} applies with $S$ the set of symbols and $U=P$.

Properly color $A$ with the two colors in $Q$. If $b$ was retained,
interchange the two colors on its connected component if needed to
give $b$ its prescribed color. This is the only prescribed color
in $A$. Transfer the colors of the vertices of $T$ to $G$ and discard
the symbol colors. This properly colors $G[L\cup T\cup B]$: the colors on $T$
are in $Q$, every edge within $T$ or from $T$ to a retained $b$
was colored in $A$, and an omitted $b$ has color $\alpha$.
Edges from $L$ to $B$ retain their prescribed colors.

For $D\in\cD$, let $a_D$ denote its possible contact in $T$ to the
left, which is unique by \ref{contact:gaps}. Let $R_D$ be its full
set of contacts in $T\cup B$ to the right. Each vertex of $T$ is
either wholly before or wholly after $D$, so these sets account for
every contact of $D$ in $T\cup B$.

If $D$ does not contact an omitted $b$, then $R_D=N_A(s_D)$.
All these vertices have the color opposite to the color of $s_D$
in the two-coloring of $A$, and hence $R_D$ is monochromatic. It is
also stable in $G$: every original edge between two of its vertices
was retained in $A$. If $D$ contacts an omitted $b$, then $b$ has
color $\alpha$. Properness of the prescribed coloring on $L\cup B$
implies that $b$ has no neighbor in $L$. By \ref{contact:right},
$R_D=\{b\}$. Thus in every case $R_D$ is stable and monochromatic;
an empty $R_D$ is allowed.

Consider the induced graph on $R_D\cup D\cup\{a_D\}$, omitting
$a_D$ if there is no left contact, and reverse its order. Its left
boundary is now $R_D$, its interior is $D$, and its optional right
boundary is $a_D$. Its boundary coloring is proper because all
edges from $a_D$ to $R_D$ are already properly colored in
$G[L\cup T\cup B]$. The graph remains triangle-free and capped.
Since $T$ is nonempty, $|D|<|X|$, so the induction hypothesis extends
this boundary coloring over $D$.

Do this for each component $D$. Different interiors are anticomplete
because they are different components of $G[X\setminus T]$.
They have no neighbors in $L$, by the definition of $T$, and all of
their contacts in $T\cup B$ were included in their boundaries.
Every extension preserves the already assigned colors, so the
extensions combine to color all of $G$. If $X=T$, no further extension is needed.
\end{proof}

Taking both boundaries empty in \cref{lem:boundary} gives the following
consequence.

\begin{corollary}\label{cor:triangle-free}
Every triangle-free capped ordered graph is three-colorable.
\end{corollary}

\section{Edge ranks and the exponential bound}\label{sec:ranks}\label{sec:exponential}

Davies, Krawczyk, McCarty and Walczak~\cite{DaviesEtAl2023} proved that the
edges of a capped graph $G$ with an edge can be partitioned into $\omega(G)-1$
triangle-free capped subgraphs. We give a rank proof of this
decomposition; the rank inequalities will also be used in the
forest construction.

Fix a capped ordered graph $G$ with an edge, and write
$w=\omega(G)$ and $q=w-1$. For each original edge $uv$ with $u<v$,
define its \emph{rank} by
\begin{equation}\label{eq:rank}
 p(uv)=\omega\bigl(G[\{x\in N_G(v):x\leq u\}]\bigr).
\end{equation}
The set in this definition contains $u$, so $p(uv)\geq1$.
Adjoining $v$ to a clique in that set gives a larger clique of $G$,
so $p(uv)\leq q$. A clique attaining $p(uv)$ will be used as a
witness for the rank. Such a clique need not contain $u$.
For a path $a<b<c$, we call $p(ab)$ and $p(bc)$ its incoming and
outgoing ranks, respectively.

\Needspace{8\baselineskip}
\begin{lemma}\label{lem:ranks}
The ranks in \eqref{eq:rank} have the following properties.
\begin{enumerate}[label=\textup{(\roman*)}]
\item\label{rank:transport}
If $ac,bd\in E(G)$ and $a\leq b<c\leq d$, then
$ad\in E(G)$ and $p(ac)\leq p(ad)\leq p(bd)$.
\item\label{rank:triangle}
If $a<b<c$ form a triangle in $G$, then
$p(bc)\geq p(ab)+1$.
\end{enumerate}
\end{lemma}

\begin{proof}
For (i), the edge $ad$ is already one of the assumed edges if $a=b$
or $c=d$. In the remaining case it follows from cappedness. Let $K$
be a clique witnessing $p(ac)$, so each $x\in K$ satisfies $x\leq a$
and $xc\in E(G)$.

If $c<d$ and $x<b$, the order $x<b<c<d$ and the edges $xc,bd$
give $xd$. The only remaining possibility is $x=b$, which forces
$x=a=b$ and uses the given edge $bd$. If $c=d$, the adjacency to
$d$ was already part of the choice of $K$. Thus $K$ is contained
in $N_G(d)\cap\{x:x\leq a\}$, proving $p(ac)\leq p(ad)$.
The inclusion
$N_G(d)\cap\{x:x\leq a\}\subseteq N_G(d)\cap\{x:x\leq b\}$
proves the second inequality.

For (ii), take a clique $K$ witnessing $p(ab)$. The same transfer
argument, applied to the edges $ab$ and $ac$, puts every vertex of
$K$ in $N_G(c)$. Each vertex of $K$ is also adjacent to $b$ and is
at most $a<b$. Hence $K\cup\{b\}$ is a clique in the prefix of
$N_G(c)$ through $b$, which proves (ii).
\end{proof}

For $1\leq h\leq q$, let $G_h$ be the spanning subgraph of $G$
whose edges have rank exactly $h$. If two edges of $G_h$ cross, \cref{lem:ranks}\ref{rank:transport} forces
their cap also to have rank $h$. Thus $G_h$ is capped.
It is triangle-free by \cref{lem:ranks}\ref{rank:triangle}, since a triangle
$a<b<c$ satisfies $p(bc)>p(ab)$.

By \cref{cor:triangle-free}, fix a proper coloring
$\varphi_h:V(G)\to\{1,2,3\}$ of each $G_h$. Give $v$ the tuple
$(\varphi_1(v),\ldots,\varphi_q(v))$. Every original edge has some
rank $h$, and its ends receive different entries in coordinate $h$.
We obtain the following bound.

\begin{proposition}\label{prop:exponential}
Every nonempty capped ordered graph $G$ satisfies
$\chi(G)\leq3^{\omega(G)-1}$.
\end{proposition}

The edgeless case needs one color. When this proposition is applied
to a different graph, its ranks are computed within that graph;
the ranks $p$ assigned to the edges of $G$ remain fixed.

\section{The cubic bound}\label{sec:cubic}\label{sec:forests}

\subsection{Selecting forests from ordered neighborhoods}\label{sec:selection}

We next select forests whose edges meet certain induced paths on
three vertices. Let $G$ be a capped ordered graph.

Let $B\subseteq V(G)$ be a set of potential centers. For each
$b\in B$, choose sets $A_b$ and $C_b$ of neighbors of $b$ such that
$A_b<b<C_b$; call them its left and right shores. These sets may be
empty, and shores of different centers may overlap. A vertex in a
shore may itself be a potential center.

\Needspace{5\baselineskip}
\begin{lemma}\label{lem:forest}
Suppose that a graph $\Gamma$ on $B$ satisfies the following two
conditions.
\begin{enumerate}[label=\textup{(\roman*)}]
\item\label{shore:incidence}
If $b,t\in B$ and $t\in A_b\cup C_b$, then $bt\in E(\Gamma)$.
\item\label{shore:crossing}
If $r<a<c<t$, where $r,t\in B$, $a\in A_t$, and $c\in C_r$,
then $rt\in E(\Gamma)$.
\end{enumerate}
Let $W\subseteq B$ be stable in $\Gamma$. Each vertex
$v\in V(G)\setminus W$ makes the following selections, when the
indicated set of centers is nonempty:
\begin{itemize}
\item it selects the least $b\in W$ with $b<v$ and $v\in C_b$;
\item it selects the greatest $b\in W$ with $v<b$ and $v\in A_b$.
\end{itemize}
Let $F_W$ be the spanning subgraph of $G$ consisting of the selected
edges. Then $F_W$ is a forest. Moreover, whenever $a<b<c$ is an
induced path in $G$ with $b\in W$, $a\in A_b$ and $c\in C_b$,
at least one of $ab,bc$ belongs to $F_W$.
\end{lemma}

\begin{proof}
No vertex of $W$ belongs to a shore of another vertex of $W$:
condition (i) would give an edge inside the stable set $W$ in
$\Gamma$. In particular, the endpoints $a,c$ of every path in the
statement are outside $W$. Each selected edge has exactly one end
in $W$ and one end outside it, and each vertex outside $W$ selects
at most one center on either side.

Suppose that $F_W$ has a cycle. It alternates between $k$ distinct
centers in $W$ and $k$ distinct vertices outside $W$, for some
$k\geq2$. Each outside vertex on the cycle has two selected cycle
edges. Since it selects at most one center on each side, it is
strictly between its two cycle centers.

Order the $k$ cycle centers. Every outside vertex on the cycle is
in one of the $k-1$ gaps between consecutive cycle centers. Since
there are $k$ such outside vertices, two of them, say $v<w$, are
in the same gap. Let $r$ be the cycle center selected by $w$ to its
left, and let $t$ be the cycle center selected by $v$ to its right.
There is no cycle center between $v$ and $w$, so $r<v<w<t$.
The selections give $v\in A_t$ and $w\in C_r$. Condition (ii)
therefore gives $rt\in E(\Gamma)$, contrary to $r,t\in W$.
Thus $F_W$ is a forest.

Now consider an induced path $a<b<c$ in the statement. The center
$b$ is available for both endpoint selections. If neither path edge
were selected, then $a$ would select a later center $t>b$, by the
greatest-center rule, and $c$ would select an earlier center $r<b$,
by the least-center rule. We claim that $t>c$ and $r<a$.

If $b<t<c$, the crossing edges $at,bc$ would give $ac$, contrary
to inducedness; if $t=c$, the selected edge $at$ would itself be
$ac$. Thus $t>c$. Similarly, if $a<r<b$, the crossing edges
$ab,rc$ would give $ac$, and $r=a$ would give $ac$ directly.
Thus $r<a$. We now have $r<a<c<t$ with $a\in A_t$ and
$c\in C_r$. Condition (ii) again gives an edge $rt$ inside $W$,
a contradiction. Hence at least one path edge was selected.
\end{proof}

\subsection{Fixed families for pairs of ranks}\label{sec:rows}

Return to the graph $G$ and the ranks $p$ fixed in
\cref{sec:ranks}. For every
pair $1\leq j\leq i\leq q$, define the full shores at a vertex $b$ by
\begin{equation}\label{eq:shores}
 \begin{aligned}
 A_b&=\{a\in V(G):a<b,\ ab\in E(G),\ p(ab)=i\},\\
 C_b&=\{c\in V(G):b<c,\ bc\in E(G),\ p(bc)=j\}.
 \end{aligned}
\end{equation}
For $a\in A_b$ and $c\in C_b$, the path $a<b<c$ is induced in
$G$. Otherwise \cref{lem:ranks}\ref{rank:triangle} would give $j\geq i+1$,
contrary to $j\leq i$.

For each pair $(i,j)$, we construct a family of forests meeting
every such path. To do so, we define a graph $\Gamma$ satisfying
the two conditions of \cref{lem:forest} and apply that lemma to
each color class of a proper coloring of $\Gamma$.
The number of forests is at most the number of colors used.

\paragraph{Equal ranks.}

When $i=j$, take $B=V(G)$ and $\Gamma=G_i$. The first condition
of \cref{lem:forest} follows directly from \eqref{eq:shores}.
For its second condition, suppose $r<a<c<t$, $a\in A_t$ and
$c\in C_r$. Both $at$ and $rc$ have rank $i$, so
\cref{lem:ranks}\ref{rank:transport} gives the edge $rt$ of rank $i$.
Use the fixed three-coloring $\varphi_i$ of $G_i$ from
\cref{sec:ranks}, including its color assignments to isolated
vertices. Applying \cref{lem:forest} in the ambient graph $G$
to its three color classes gives at most three forests.
These meet every induced path $a<b<c$ with $p(ab)=p(bc)=i$.

\paragraph{Strictly decreasing ranks.}

Fix $i>j$, and now take
$B=\{b\in V(G):A_b\ne\varnothing,\ C_b\ne\varnothing\}$.
Only these centers occur in the paths to be covered. For $b\in B$,
let $\alpha_b=\min A_b$ and $\beta_b=\max C_b$. We first record
two consequences of the rank inequalities.

If $x<y$ are in $B$, $xy\in E(G)$ and $j<p(xy)<i$, then
\begin{equation}\label{eq:barriers}
 \beta_x<y\qquad\text{and}\qquad x<\alpha_y.
\end{equation}
Indeed, $y\leq\beta_x$ would give
$p(xy)\leq p(x\beta_x)=j$ by \cref{lem:ranks}\ref{rank:transport}. Likewise,
$\alpha_y\leq x$ would give
$i=p(\alpha_y y)\leq p(xy)$. Both contradict the strict rank
assumption.

Also, if $x,y\in B$ and $x<y<\beta_x$, then
$\beta_y\leq\beta_x$.
If instead $\beta_x<\beta_y$, the rank-$j$ edges
$x\beta_x$ and $y\beta_y$ cross. Their cap $x\beta_y$ has rank
$j$ by \cref{lem:ranks}\ref{rank:transport}, so $\beta_y\in C_x$. This contradicts
the maximality of $\beta_x$.

Define $\Gamma$ on $B$ as follows. An original edge $xy$, $x<y$,
belongs to $\Gamma$ if its rank is $j$ or $i$, or if its rank is
strictly between $j$ and $i$ and
\begin{equation}\label{eq:overlap}
 x<\alpha_y<\beta_x<y.
\end{equation}
No other edges belong to $\Gamma$. For a strictly intermediate rank, the inequalities
$x<\alpha_y$ and $\beta_x<y$ already follow from
\eqref{eq:barriers}.

\begin{lemma}\label{lem:center}
The graph $\Gamma$ is capped and has clique number at most four.
It also satisfies both conditions of \cref{lem:forest} for the
shores in \eqref{eq:shores} and the ambient graph $G$.
\end{lemma}

\begin{proof}
\emph{Cappedness.}
Take edges $xz,yt\in E(\Gamma)$ with $x<y<z<t$.
Their cap $xt$ is an original edge, and if $h=p(xt)$, then
\cref{lem:ranks}\ref{rank:transport} gives
$j\leq p(xz)\leq h\leq p(yt)\leq i$.
If $h=j$ or $h=i$, the definition includes $xt$ in $\Gamma$.
Assume $j<h<i$. By \eqref{eq:barriers},
$\beta_x<t$ and $x<\alpha_t$. It remains to prove
$\alpha_t<\beta_x$.

First suppose $y<\beta_x$. Since $p(yt)\geq h>j$, the edge
$yt$ either has rank $i$ or has a strictly intermediate rank.
In the former case $y\in A_t$, so
$\alpha_t\leq y<\beta_x$. In the latter case $yt\in E(\Gamma)$
gives $\alpha_t<\beta_y$ by \eqref{eq:overlap}; and the observation for
$x<y<\beta_x$ gives $\beta_y\leq\beta_x$.

It remains to consider $\beta_x\leq y$. We have $p(xz)\leq h<i$.
The rank of $xz$ cannot be $j$, since then $z\in C_x$ would give
$z\leq\beta_x\leq y$, contrary to $y<z$. Therefore $xz$ has
a strictly intermediate rank and satisfies \eqref{eq:overlap}.
In particular,
$\alpha_z<\beta_x\leq y<z<t$. Applying
\cref{lem:ranks}\ref{rank:transport} to the original edges $\alpha_z z$ and
$yt$ gives an edge $\alpha_z t$ and
$i=p(\alpha_z z)\leq p(\alpha_z t)\leq p(yt)\leq i$.
Thus $\alpha_z\in A_t$ and
$\alpha_t\leq\alpha_z<\beta_x$. This case includes
$\beta_x=y$. In both cases \eqref{eq:overlap} holds for $xt$,
so $xt\in E(\Gamma)$.

\emph{The clique bound.}
Suppose $v_1<\cdots<v_5$ form a clique in $\Gamma$.
Every one of its edges is an original edge of rank between $j$
and $i$. The triangle $v_1,v_2,v_3$ gives
$p(v_2v_3)\geq p(v_1v_2)+1>j$.
The triangle $v_3,v_4,v_5$ gives
$p(v_3v_4)\leq p(v_4v_5)-1<i$.
Using \cref{lem:ranks}\ref{rank:transport} with a common left or right endpoint,
we obtain
$j<p(v_2v_3)\leq p(v_2v_4)\leq p(v_3v_4)<i$.
All three middle edges therefore satisfy \eqref{eq:overlap}.
The edges $v_2v_3$ and $v_3v_4$ give
$\beta_{v_2}<v_3<\alpha_{v_4}$, whereas $v_2v_4$ gives
$\alpha_{v_4}<\beta_{v_2}$. This contradiction proves
$\omega(\Gamma)\leq4$.

\emph{The two selection conditions.}
If another center is in a shore of $b$, its edge to $b$ has rank
$i$ or $j$, and hence is included in $\Gamma$. This proves the
first condition of \cref{lem:forest}. For the second, suppose
$r<a<c<t$, $a\in A_t$, and $c\in C_r$.
The edges $rc$ and $at$ give $rt\in E(G)$ and
$j\leq p(rt)\leq i$ by \cref{lem:ranks}\ref{rank:transport}.
Again the endpoint ranks $j,i$ give $rt\in E(\Gamma)$ directly.
For a strictly intermediate rank, \eqref{eq:barriers} gives
$\beta_r<t$ and $r<\alpha_t$, while the shore definitions give
$\alpha_t\leq a<c\leq\beta_r$. Together these inequalities
are precisely $r<\alpha_t<\beta_r<t$, so $rt\in E(\Gamma)$.
\end{proof}

By \cref{prop:exponential,lem:center}, $\Gamma$ has a proper
coloring with at most $3^{4-1}=27$ colors. If $\Gamma$ has no vertices,
no colors are needed; if it is nonempty and edgeless, one color suffices. Apply
\cref{lem:forest} to each color class. This gives at most $27$ forests meeting every induced path
$a<b<c$ with $p(ab)=i$ and $p(bc)=j$. Every selected edge
has rank $i$ or $j$, by the definition of the shores.

\paragraph{Grouping the forests by incoming rank.}

For each pair $(i,j)$, fix one of the colorings just described and
all the resulting selections. Write $\cF_{i,j}$ for this family
of forests. These choices are made in $G$. Distinct families may share edges,
and a forest may have no edges.

For $1\leq r\leq q$, let $\cR_r$ be the collection of all forests
in $\cF_{r,1},\ldots,\cF_{r,r}$, and let $R_r$ be the spanning
graph whose edge set is the union of the edge sets of these forests. We call
$\cR_r$ the $r$th row of forests. It contains at most
$3+27(r-1)\leq27r$ forests. The preceding construction gives
the following two properties:
\begin{align}
 e\in E(R_r)&\quad\Longrightarrow\quad p(e)\leq r;
 \label{eq:row-rank}\\
 \left.\begin{gathered}
 a<b<c\text{ is an induced path in }G,\\
 p(ab)=r\geq p(bc)
 \end{gathered}\right\}
 &\quad\Longrightarrow\quad
 \{ab,bc\}\cap E(R_r)\ne\varnothing.
 \label{eq:row-cover}
\end{align}
The first property holds because each selected edge has rank
$r$ or rank $j\leq r$. For the second, take
$j=p(bc)$ and the family $\cF_{r,j}$. The center $b$ belongs
to one of its color classes, and the forest for that class
meets the path by \cref{lem:forest}.

Let $F$ be the spanning graph with
$E(F)=\bigcup_{r=1}^q E(R_r)$, and let $H$ be the spanning
graph with $E(H)=E(G)\setminus E(F)$. All the ranks and selections are now fixed. Every subsequent
restriction to a vertex set uses these same forests.

\subsection{The coloring of the original graph}\label{sec:main-proof}

For each original vertex $v$, define
\begin{equation}\label{eq:m}
 m(v)=\max\bigl(\{p(xv):x<v,\ xv\in E(H)\}\cup\{0\}\bigr).
\end{equation}
Thus $0\leq m(v)\leq q$. It is the largest rank of an edge of $H$ entering $v$ from the
left, or zero if there is no such edge. The next lemma relates
$m(v)$ to the selected rows.

\Needspace{5\baselineskip}
\begin{lemma}\label{lem:forcing}
If $uv\in E(G)$, $u<v$, and $p(uv)\leq m(u)$, then
$uv\in E(R_{m(u)})$.
\end{lemma}

\begin{proof}
Since $p(uv)\geq1$, the hypothesis gives $m(u)>0$.
The maximum in \eqref{eq:m} is therefore attained by an
incoming edge $xu\in E(H)$ with $x<u$ and $p(xu)=m(u)$.
If $xv$ were an original edge, the triangle $x<u<v$ and
\cref{lem:ranks}\ref{rank:triangle} would give
$p(uv)\geq p(xu)+1=m(u)+1$, a contradiction.
Hence $x<u<v$ is an induced path in $G$, and its incoming
rank $m(u)$ is at least its outgoing rank $p(uv)$.

By \eqref{eq:row-cover}, one of $xu,uv$ belongs to
$R_{m(u)}$. Since $xu\in E(H)$ and $E(H)\cap E(F)=\varnothing$,
it follows that $uv\in E(R_{m(u)})$.
\end{proof}

For every edge $uv\in E(H)$ with $u<v$, the conclusion of
\cref{lem:forcing} is impossible. Thus
$m(u)<p(uv)\leq m(v)$, where the second inequality follows
from the definition of $m(v)$. Consequently $v\mapsto m(v)+1$ properly colors $H$.

For $s=0,\ldots,q$, let $U_s=\{v\in V(G):m(v)=s\}$.
These sets partition $V(G)$ and are stable in $H$.
The following lemma locates the edges of $G[U_s]$ in the selected rows.

\begin{lemma}\label{lem:suffix}
For every $s\in\{0,\ldots,q\}$,
\begin{equation}\label{eq:suffix}
 E(G[U_s])\subseteq
 \bigcup_{r=\max\{1,s\}}^q E(R_r[U_s]).
\end{equation}
\end{lemma}

\begin{proof}
Take an original edge $uv$ within $U_s$, with $u<v$, and let
$h=p(uv)$. Since $U_s$ is stable in $H$, this edge belongs to
$F$. If $h\leq s=m(u)$, \cref{lem:forcing} puts it in
$R_s$. This case can occur only when $s\geq1$, because
all edge ranks are positive.

If $h>s$, choose any row $R_r$ containing $uv$; such a row
exists because $uv\in E(F)$. By \eqref{eq:row-rank},
$h\leq r$, so this row has $r>s$ and belongs to the union
in \eqref{eq:suffix}. This proves the inclusion. When $s=0$
only this second case occurs, and the union begins at row $1$.
\end{proof}

\begin{proof}[Proof of \cref{thm:main}]
The nonempty edgeless case was noted in
\cref{sec:introduction}. Assume that $G$ has an edge, and
use the fixed construction above. For $s=0,\ldots,q$, let
$d_s=\sum_{r=\max\{1,s\}}^q27r$.
The rows appearing in \eqref{eq:suffix} contain at most
$d_s$ forests in total. Each is still a forest when restricted
to any vertex set.

Take any nonempty $W\subseteq U_s$. A forest on $W$ has at
most $|W|-1$ edges. Restricting \eqref{eq:suffix} to $W$ and summing over the forests
therefore gives $|E(G[W])|\leq d_s(|W|-1)$,
whether or not their edge sets are disjoint. Since $q\geq1$, the sum
defining $d_s$ includes row $q$, and $d_s\geq1$. Hence the
average degree of $G[W]$ is at most
$2d_s(|W|-1)/|W|<2d_s$.
There is a vertex of degree at most $2d_s-1$ in $G[W]$.
This is true for every nonempty induced subgraph of $G[U_s]$.
Repeatedly remove such a vertex until none remain, and then
restore the vertices in the reverse order. When a vertex is
restored, at most $2d_s-1$ of its neighbors have been colored,
so one of $2d_s$ colors is available. This gives a proper
coloring of the whole original graph $G[U_s]$ with at most
$2d_s$ colors. An empty $U_s$ needs no colors.

Use disjoint palettes for different $U_s$. Each $G[U_s]$ is
properly colored, and the disjoint palettes separate the ends
of every edge between different sets $U_s$. Thus this is a
proper coloring of $G$.

\Needspace{10\baselineskip}
It remains to count the colors. For a fixed row index
$r\in\{1,\ldots,q\}$, its summand $27r$ occurs in
$d_s$ exactly when $s\in\{0,1,\ldots,r\}$, so it is counted
$r+1$ times. Therefore
\begin{align*}
 \chi(G)&\leq2\sum_{s=0}^q d_s
 =54\sum_{r=1}^q r(r+1)\\*
 &=18q(q+1)(q+2)
 =18(w^3-w)\leq18w^3.
\end{align*}
This calculation also holds when $q=1$, where there are no
strictly decreasing rank pairs. The independent coloring
in \cref{prop:exponential} gives $3^q$ colors. Choosing the
better of these two proper colorings proves the stated
minimum bound.
\end{proof}

\section*{Declaration of competing interest}

The author declares that he has no known competing financial interests
or personal relationships that could have appeared to influence the
work reported in this paper.

\section*{Acknowledgements}

The research was partially supported by RGC Competitive Earmarked
Research Grant 16308821.

\section*{Data availability}

No data was used for the research described in the article.

\section*{Declaration of AI assistance}

ChatGPT assisted with  the direct elimination argument in the proof of
\cref{lem:bipartite}. It also assisted with language editing. The author takes full responsibility for the accuracy and content of this work.

\end{document}